\documentclass[preprint,3p,10pt,authoryear]{elsarticle}

\journal{Journal of Mathematical Analysis and Applications}

\usepackage{amssymb}
\usepackage{amsthm}
\usepackage{lineno}  

\usepackage{enumerate,graphicx,mathtools,natbib,geometry,xcolor}
\usepackage[colorlinks]{hyperref}

\newtheorem{theorem}{Theorem}[section]
\newtheorem{proposition}[theorem]{Proposition}
\newtheorem{definition}[theorem]{Definition}
\newtheorem{corollary}[theorem]{Corollary}
\newtheorem{lemma}[theorem]{Lemma}
\newtheorem{remark}[theorem]{Remark}

\makeatletter \@addtoreset{equation}{section} \makeatother

\newcommand{\N}{\mathbb{N}}
\newcommand{\R}{\mathbb{R}}
\newcommand{\PP}{\mathbb{P}}

\newcommand{\bb}[1]{\boldsymbol{#1}}

\begin{document}

\begin{frontmatter}

    \title{Complete monotonicity of multinomial probabilities and its application to Bernstein estimators on the simplex}

    \author[a1]{Fr\'ed\'eric Ouimet\corref{cor1}\fnref{fn1}}

    \address[a1]{2920, chemin de la Tour, Universit\'e de Montr\'eal, Montr\'eal, QC H3T 1J8, Canada.}

    \cortext[cor1]{Corresponding author}
    \ead{ouimetfr@dms.umontreal.ca}

    \fntext[fn1]{F. Ouimet is supported by a NSERC Doctoral Program Alexander Graham Bell scholarship.}

    \begin{abstract}
        Let $d\in \N$ and let $\gamma_i\in [0,\infty)$, $x_i\in (0,1)$ be such that $\sum_{i=1}^{d+1} \gamma_i  = M\in (0,\infty)$ and $\sum_{i=1}^{d+1} x_i = 1$.
        We prove that
        \begin{equation*}
            a \mapsto \frac{\Gamma(aM + 1)}{\prod_{i=1}^{d+1} \Gamma(a \gamma_i + 1)} \prod_{i=1}^{d+1} x_i^{a\gamma_i}
        \end{equation*}
        is completely monotonic on $(0,\infty)$.
        This result generalizes the one found by \cite{MR3730425} for binomial probabilities ($d=1$).
        As a consequence of the log-convexity, we obtain some combinatorial inequalities for multinomial coefficients.
        We also show how the main result can be used to derive asymptotic formulas for quantities of interest in the context of statistical density estimation based on Bernstein polynomials on the $d$-dimensional simplex.
    \end{abstract}

    \begin{keyword}
        multinomial probability \sep complete monotonicity \sep Gamma function \sep combinatorial inequalities \sep Bernstein polynomials \sep simplex
        \MSC[2010]{Primary : 60C05 \sep Secondary : 62G05 \sep 62G07 \sep 33B15}
    \end{keyword}

\end{frontmatter}

\section{Introduction}\label{sec:intro}

For any $d\in \N$, let $[d] \circeq \{1,2,\ldots,d\}$. For any $\bb{v} \circeq (v_1,v_2,\ldots,v_d)\in \R^d$, write $\|\bb{v}\| \circeq \sum_{i=1}^d |v_i|$.
Denote the $d$-dimensional simplex and its interior by
\begin{equation*}
    \mathcal{S} \circeq \big\{\bb{x}\in [0,1]^d : \|\bb{x}\| \leq 1\big\} \quad \text{and} \quad \text{Int}(\mathcal{S}) \circeq \big\{\bb{x}\in (0,1)^d : \|\bb{x}\| < 1\big\}.
\end{equation*}
Given a random sample $\bb{y}_1,\bb{y}_2,\ldots,\bb{y}_n$ on $\mathcal{S}$ from some unknown distribution $F$, define the Bernstein estimator on the simplex
\begin{equation}\label{eq:Bernstein.estimator}
    \hat{F}_{m,n}(\bb{x}) \circeq \sum_{\bb{k}\in \N_0^d :\|\bb{k}\|\leq m} F_n(\bb{k}/m) P_{\bb{k},m}(\bb{x}), \quad \bb{x}\in \mathcal{S},
\end{equation}
where $m,n\in \N$, $F_n(\bb{y}) \circeq \frac{1}{n} \sum_{j=1}^n 1_{\{\bb{y} \leq \bb{y}_j\}}$ is the empirical cumulative distribution function, $x_{d+1} \circeq 1 - \|\bb{x}\|$, $k_{d+1} \circeq m - \|\bb{k}\|$, and
\begin{equation}\label{eq:multinomial.probability}
    P_{\bb{k},m}(\bb{x}) \circeq \frac{m!}{\prod_{i=1}^{d+1} k_i!} \prod_{i=1}^{d+1} x_i^{k_i}.
\end{equation}

Our first goal is to prove that $a\mapsto P_{a\bb{k},a m}(\bb{x})$ is completely monotonic on $(0,\infty)$, see Definition \ref{def:complete.monotonicity} below.
In fact, we prove a slightly more general statement in Theorem \ref{thm:completely.monotonic}. From the log-convexity, we deduce some combinatorial inequalities for multinomial coefficients in Section \ref{sec:combinatorial.inequalities}. The proof of the theorem and the combinatorial inequalities follow very closely, and generalize, the work of \cite{MR3730425}. In Section \ref{sec:application.Bernstein.polynomials.on.simplex}, we show how Theorem \ref{thm:completely.monotonic} can be used to prove asymptotic formulas for quantities of interest related to \eqref{eq:Bernstein.estimator}. To our knowledge, the statistical properties (bias, variance, mean integrated squared error, etc.) of the estimator in \eqref{eq:Bernstein.estimator} (and the associated density estimator, see e.g.\hspace{-0.3mm} \cite{MR2270097,MR2662607}) have never been studied when $d > 1$, except for the pointwise mean squared error of the density estimator in \cite{MR1293514} when $d = 2$.
This was our motivation for this article.

\begin{definition}[Complete monotonicity]\label{def:complete.monotonicity}
    A non-constant function $a \mapsto g(a)$ is said to be completely monotonic on $(0,\infty)$, if $g$ has derivatives of all orders and satisfies
    \begin{equation}\label{eq:def:complete.monotonicity}
        (-1)^n g^{(n)}(a) > 0, \quad \text{for all } n\in \N_0, ~a\in (0,\infty).
    \end{equation}
\end{definition}

\begin{remark}
    Inequality \eqref{eq:def:complete.monotonicity} is usually not strict when defining complete monotonicity, but non-constant functions that satisfy the non-strict version of \eqref{eq:def:complete.monotonicity} automatically satisfy the strict version, see \cite[p.98]{MR0000436} for the original proof or \cite[p.395]{MR1421454} for a simpler proof.
\end{remark}

We will need the two following lemmas during the proof of Theorem \ref{thm:completely.monotonic}.

\begin{lemma}\label{lem:Alzer.2018.lemma.2}
    Let $g : (0,\infty) \to (0,1)$.
    If $(-\log g)'$ is completely monotonic on $(0,\infty)$, then $g$ is completely monotonic on $(0,\infty)$.
\end{lemma}

\begin{proof}
    Take $f : (0,\infty) \to (0,1) : y \mapsto e^{-y}$ and $h : (0,\infty) \to (0,\infty) : x \mapsto - \log g(x)$.
    Since $h$ is positive and $h' = (-\log g)'$ is completely monotonic by hypothesis, then $g = f \circ h$ is completely monotonic by Theorem 2 in \cite{MR1872377}.
\end{proof}

\begin{lemma}\label{lem:Alzer.2018.lemma.1.generalization}
    If $\bb{u} \circeq (u_1,u_2,\ldots,u_d)\in \text{Int}(\mathcal{S})$, $u_{d+1} \circeq 1 - \|\bb{u}\| > 0$ and $y > 1$, then
    \begin{equation}\label{eq:lem:Alzer.2018.lemma.1.generalization}
        J_{\bb{u}}(y) \circeq \frac{1}{y - 1} - \sum_{i=1}^{d+1} \frac{1}{y^{1/u_i} - 1} > 0.
    \end{equation}
\end{lemma}

\begin{proof}
    Lemma 1 in \cite{MR3730425} proves \eqref{eq:lem:Alzer.2018.lemma.1.generalization} in the case $d = 1$.
    Fix $d\geq 2$ and assume that \eqref{eq:lem:Alzer.2018.lemma.1.generalization} is true for any smaller integer.
    Let $y > 1$. By Lemma 1 in \cite{MR3730425},
    \begin{equation}
        \frac{1}{y - 1} - \frac{1}{y^{1/\|\bb{u}\|} - 1} - \frac{1}{y^{1/(1 - \|\bb{u}\|)} - 1} > 0.
    \end{equation}
    Therefore, \eqref{eq:lem:Alzer.2018.lemma.1.generalization} will follow if we can show that
    \begin{equation}\label{eq:lem:Alzer.2018.lemma.1.generalization.sufficient.equation}
        \frac{1}{y^{1/\|\bb{u}\|} - 1} - \sum_{i=1}^d \frac{1}{y^{1/u_i} - 1} > 0.
    \end{equation}
    Simply define $z \circeq y^{1/\|\bb{u}\|}$ and $v_i \circeq u_i / \|\bb{u}\|$, then \eqref{eq:lem:Alzer.2018.lemma.1.generalization.sufficient.equation} is equivalent to
    \begin{equation}
        \frac{1}{z - 1} - \sum_{i=1}^d \frac{1}{z^{1/v_i} - 1} > 0,
    \end{equation}
    which is true by the induction hypothesis.
\end{proof}

\section{Main result}\label{sec:main.result}

Below is a generalization of the theorem in \cite{MR3730425}.

\begin{theorem}\label{thm:completely.monotonic}
    For any $d\in \N$, $M\in (0,\infty)$, $\bb{x}\in \text{Int}(\mathcal{S})$, $x_{d+1} \circeq 1 - \|\bb{x}\| > 0$, and any $\bb{\gamma}\in [0,\infty)^d$ such that $\|\bb{\gamma}\| \leq M$ and $\gamma_{d+1} \circeq M - \|\bb{\gamma}\| \geq 0$, the function
    \begin{equation}\label{eq:thm:completely.monotonic}
        g(a) \circeq \frac{\Gamma(aM + 1)}{\prod_{i=1}^{d+1} \Gamma(a \gamma_i + 1)} \prod_{i=1}^{d+1} x_i^{a\gamma_i}
    \end{equation}
    is completely monotonic on $(0,\infty)$.
\end{theorem}

\begin{remark}
    In the proof of Theorem \ref{thm:completely.monotonic}, we will show that $(-\log g)'$ is completely monotonic on $(0,\infty)$, which is a stronger statement by Lemma \ref{lem:Alzer.2018.lemma.2}.
\end{remark}

\begin{remark}
    Soon after the first version of the present paper was posted on arXiv.org, \cite{hal-01769288} gave an alternative proof of the complete monotonicity of $(-\log g)'$ and rewrote the combinatorial inequalities of Section \ref{sec:combinatorial.inequalities} in terms of multivariate beta functions.
\end{remark}

\begin{proof}
    Let $M\in (0,\infty)$, $\bb{x}\in \text{Int}(\mathcal{S})$ and $a > 0$.
    The theorem in \cite{MR3730425} proves our statement in the case $d = 1$ (when the components of $\bb{\gamma}$ are integers, but the adjustment is trivial).
    Therefore, fix $d\geq 2$ and assume that the theorem is true for any smaller integer.
    If there exists $i\in [d+1]$ such that $\gamma_i = 0$, the theorem reduces to proving that \eqref{eq:thm:completely.monotonic} is completely monotonic for a $d$ that is smaller then the one that we previously fixed, which is true by the induction hypothesis.
    Thus, assume for the remainder of the proof that
    \begin{equation}\label{eq:prop:completely.monotonic.hypothesis.k.i}
        \gamma_i > 0, \quad \text{for all } i\in [d+1].
    \end{equation}

    Define
    \begin{equation}\label{eq:prop:completely.monotonic.h}
        h(a) \circeq -\log g(a) = -\log \Gamma(aM + 1) + \sum_{i=1}^{d+1} \log \Gamma(a \gamma_i + 1) - a \sum_{i=1}^{d+1} \gamma_i \log x_i.
    \end{equation}
    Then,
    \begin{equation}\label{eq:prop:completely.monotonic.derivative.h}
        h'(a) = - M\psi(aM+1) + \sum_{i=1}^{d+1} \gamma_i \psi(a \gamma_i + 1) - \sum_{i=1}^{d+1} \gamma_i \log x_i,
    \end{equation}
    where $\psi \circeq (\log \Gamma)' = \Gamma' / \Gamma$.
    Using the integral representation
    \begin{equation}\label{eq:Abra.Stegun.1}
        \psi'(z) = \int_0^{\infty} \frac{te^{-(z - 1)t}}{e^t - 1} dt, \quad z > 0,
    \end{equation}
    see \cite[p.260]{MR0167642}, we obtain (take $t = s/M$ and $t = s/\gamma_i$)
    \begin{align}
        h''(a)
        &= - M^2 \psi'(aM + 1) + \sum_{i=1}^{d+1} \gamma_i^2 \psi'(a \gamma_i + 1)
        = - M^2 \int_0^{\infty} \frac{te^{-aM t}}{e^t - 1} dt + \sum_{i=1}^{d+1} \gamma_i^2 \int_0^{\infty} \frac{te^{-a\gamma_i t}}{e^t - 1} dt \notag \\
        &= - \int_0^{\infty} se^{-as} J_{\bb{\gamma}/M}(e^{s/M}) ds,
    \end{align}
    where $J_{\bb{u}}(y)$ is defined in \eqref{eq:lem:Alzer.2018.lemma.1.generalization}.
    Applying Lemma \ref{lem:Alzer.2018.lemma.1.generalization} gives
    \begin{equation}\label{eq:prop:completely.monotonic.almost.equation}
        (-1)^n h^{(n+1)}(a) = \int_0^{\infty} s^n e^{-as} J_{\bb{\gamma}/M}(e^{s/M}) ds > 0, \quad n\in \N, ~a > 0.
    \end{equation}
    If we show that $h'(a) > 0$ for $a > 0$, then $h'$ will be completely monotonic under Definition \ref{def:complete.monotonicity} and we will be able to conclude that $g$ is completely monotonic by Lemma \ref{lem:Alzer.2018.lemma.2}. Since $h'$ is decreasing (see \eqref{eq:prop:completely.monotonic.almost.equation} when $n=1$), we show that $\lim_{a\to\infty} h'(a) \geq 0$ to conclude the proof.

    If we apply the recurrence formula
    \begin{equation}\label{Abra.Stegun.2}
        \psi(z + 1) = \psi(z) + \frac{1}{z}, \quad z > 0,
    \end{equation}
    see \cite[p.258]{MR0167642}, we obtain from \eqref{eq:prop:completely.monotonic.derivative.h} the representation
    \begin{equation}\label{eq:prop:completely.monotonic.derivative.h.decomposition}
        h'(a) = \frac{d}{a} - M R(aM) + \sum_{i=1}^{d+1} \gamma_i R(a \gamma_i) + \sum_{i=1}^{d+1} \gamma_i \log\left(\frac{\gamma_i / M}{x_i}\right),
    \end{equation}
    where $R(z) \circeq \psi(z) - \log z$. Using the asymptotic formula
    \begin{equation}\label{eq:Abra.Stegun.3}
        \psi(z) \sim \log z - \frac{1}{2z} - \ldots \quad (\text{as } z\to \infty)
    \end{equation}
    see \cite[p.259]{MR0167642}, we conclude from \eqref{eq:prop:completely.monotonic.derivative.h.decomposition} and Jensen's inequality (for the convex function $-\log(\cdot)$ and the probability weights $P_i \circeq \gamma_i / M$ and $Q_i \circeq x_i$) that
    \begin{equation}\label{eq:prop:completely.end}
        \lim_{a\to\infty} h'(a) = M \sum_{i=1}^{d+1} \frac{\gamma_i}{M} \log\left(\frac{\gamma_i / M}{x_i}\right) \geq - M \log\left(\sum_{i=1}^{d+1} x_i\right) = 0.
    \end{equation}
    This ends the proof.
\end{proof}

\begin{remark}
    Interestingly, the sum on the left-end side of the inequality in \eqref{eq:prop:completely.end} is the Kullback-Leibler divergence $D_{\text{KL}}(P\|Q)$.
    It is well defined because of \eqref{eq:prop:completely.monotonic.hypothesis.k.i} and the fact that $\bb{x}\in \text{Int}(\mathcal{S})$ by hypothesis (which implies $0 < x_i < 1$ for all $i\in [d+1]$).
\end{remark}

\section{Some combinatorial inequalities}\label{sec:combinatorial.inequalities}

In the context of Theorem \ref{thm:completely.monotonic}, define
\begin{equation}\label{eq:multinomial.coefficients.general}
    C(a) \circeq \frac{\Gamma(aM + 1)}{\prod_{i=1}^{d+1} \Gamma(a \gamma_i + 1)}, \quad a\in (0,\infty).
\end{equation}
Below are three simple combinatorial inequalities for the multinomial coefficients in \eqref{eq:multinomial.coefficients.general}.
They generalize the ones proved in \cite{MR3730425} for binomial coefficients.

\begin{corollary}\label{cor:combinatorial.inequality}
    Let $k\in \N$ and let $a_j\in (0,\infty)$, $\lambda_j\in (0,1)$, $j\in \{1,2,\ldots,k\}$, be such that $\sum_{j=1}^k \lambda_j = 1$.
    The following inequalities hold :
    \begin{enumerate}[(a)]
        \item $C(\sum_{j=1}^k \lambda_j a_j) \leq \prod_{j=1}^k C(a_j)^{\lambda_j}$, where equality holds if and only if all the $a_j$'s are the same.\vspace{-0.8mm}
        \item $\prod_{j=1}^k C(a_j) < C(\sum_{j=1}^k a_j)$.
        \item If $a_1 \leq a_3$, then $C(a_1 + a_2) C(a_3) \leq C(a_1) C(a_2 + a_3)$, where equality holds if and only if $a_1 = a_3$.
    \end{enumerate}
\end{corollary}

\begin{proof}
    By \eqref{eq:prop:completely.monotonic.almost.equation} in the case $n = 1$, we know that $g$ in the statement of Theorem \ref{thm:completely.monotonic} is strictly log-convex,    which implies $(a)$ by definition. Point $(b)$ follows from Lemma 3 in \cite{MR3730425} because $g$ is differentiable on $[0,\infty)$, $g(0) = 1$ and $g$ is (strictly) positive, (strictly) decreasing and strictly log-convex on $(0,\infty)$. Point $(c)$ follows from a trivial adaptation of the proof of Corollary 3 in \cite{MR3730425} using \eqref{eq:prop:completely.monotonic.almost.equation}.
\end{proof}

\section{Application to Bernstein estimators on the simplex}\label{sec:application.Bernstein.polynomials.on.simplex}

In recent years, there has been a sustained interest in the study of statistical properties of Bernstein estimators on the unit hypercube, whether we talk about the cumulative distribution function (cdf) estimators
\begin{equation}\label{eq:Bernstein.cumulative.distribution.estimator.unit.hypercube}
    \hat{F}_{m,n}(\bb{x}) = \sum_{\bb{k}\in \N_0^d \cap [0,m]^d} F_n(\bb{k}/m) \prod_{i=1}^d {m \choose k_i} x_i^{k_i} (1 - x_i)^{k_i}, \quad \bb{x}\in [0,1]^d,
\end{equation}
where $F_n$ denotes the empirical cdf (given a random sample $\bb{y}_1, \bb{y}_2, \ldots, \bb{y}_n$ from an unknown cdf $F$), or the density estimators
\begin{equation}\label{eq:Bernstein.density.estimator.unit.hypercube}
    \hat{f}_{m,n}(\bb{x}) = m^d \sum_{\bb{k}\in \N_0^d \cap [0,m-1]^d} \PP_n\hspace{-0.5mm}\left(\Big(\frac{\bb{k}}{m},\frac{\bb{k}+\bb{1}}{m}\Big]\right) \prod_{i=1}^d {m-1 \choose k_i} x_i^{k_i} (1 - x_i)^{k_i}, \quad \bb{x}\in [0,1]^d,
\end{equation}
where $\PP_n$ denotes the empirical measure.
For more information, the reader is referred to
\cite{MR1910059}, \cite{MR2270097}, \cite{MR3474765}, \cite{MR3630225}, \cite{MR1873330}, \cite{MR3174309}, \cite{MR2769276}, \cite{MR2879763,MR3147339,MR3668546}, \cite{MR2345922}, \cite{MR2488150,MR2662607,MR2960952,MR2925964}, \cite{MR3412755}, \cite{MR1703623}, \cite{MR2153833}, \cite{MR1293514} and \cite{MR0397977}.

One clear advantage of Bernstein estimators over kernel estimators (for example) is that they generally perform better near the boundary, see e.g.\hspace{-0.3mm} \cite{MR2925964}.
To our knowledge, the statistical properties of Bernstein estimators on the simplex (see \eqref{eq:Bernstein.estimator}), and the associated density estimators, have never been studied in the literature, except in the univariate case where they coincide with \eqref{eq:Bernstein.cumulative.distribution.estimator.unit.hypercube} and \eqref{eq:Bernstein.density.estimator.unit.hypercube} above, and except for the pointwise mean squared error of the density estimator in \cite{MR1293514} when $d = 2$.
This subject is worth investigating because there are instances in practice where the distribution that we would like to estimate lives naturally on the $d$-dimensional simplex. One such example is the Dirichlet distribution, which is the conjugate prior of the multinomial distribution in Bayesian estimation, see e.g.\hspace{-0.3mm} \cite{Lange_1995} for an application in the context of allele frequency estimation in genetics. In those instances, we would expect that the estimators defined on the simplex perform better than the ones defined on the unit hypercube, especially near the boundary $\|\bb{x}\| = 1$.

Following \cite{MR2345922} and \cite{MR2662607}, define
\begin{equation*}
    S_{r,s,m}(\bb{x}) \circeq \sum_{\bb{k}\in \N_0^d : \|\bb{k}\| \leq m} P_{r\bb{k},rm}(\bb{x}) P_{s\bb{k},sm}(\bb{x}), \quad \bb{x}\in \mathcal{S},
\end{equation*}
for $r,s,m\in \N$.
This family of polynomials would arise in the context of statistical density estimation based on the Bernstein estimators in \eqref{eq:Bernstein.estimator} (see e.g.\hspace{-0.3mm} the appendix in \cite{MR2662607}).
Theorem \ref{thm:completely.monotonic} will be used to prove Proposition \ref{prop:lemma.4.Leblanc.2006.tech.report} below.

\newpage
The following lemma generalizes Theorem 1.1 (iii) in \cite{MR2345922}, and Lemma 3 $(ii)$ and $(iv)$ in \cite{MR2662607} when $j = 0$.

\begin{lemma}\label{lem:lemma.3.Leblanc.2006.tech.report}
    Let $d,r,s,m \in \N$, $\bb{x}\in \text{Int}(\mathcal{S})$, and define the covariance matrix
    \begin{equation}\label{eq:lem:lemma.3.Leblanc.2006.tech.report.covariance.matrix}
        \bb{\Sigma} \circeq rs(r+s)(\text{diag}(\bb{x}) - \bb{x}\bb{x}^T).
    \end{equation}
    We have
    \begin{equation*}
        m^{d/2} S_{r,s,m}(\bb{x}) = \phi_{r,s}(\bb{x}) + o_{\bb{x}}(1), \quad \text{as } m\to\infty,
    \end{equation*}
    where
    \begin{equation}\label{eq:lem:lemma.3.Leblanc.2006.tech.report.psi.function}
        \phi_{r,s}(\bb{x}) \circeq \frac{(\text{gcd}(r,s))^d}{(2\pi)^{d/2} (\det(\bb{\Sigma}))^{1/2}}.
    \end{equation}
\end{lemma}

\begin{proof}
    Let $\bb{U}_1,\ldots,\bb{U}_m$ and $\bb{V}_1,\ldots,\bb{V}_m$ be two (independent) sequences of independent random vectors such that $\bb{U}_i \sim \text{Multinomial}(r,\bb{x})$ and $\bb{V}_i \sim \text{Multinomial}(s,\bb{x})$ for each $i\in [d]$.
    Now, let $\bb{H} \circeq \text{gcd}(r,s) \bb{I}_d$ where $\bb{I}_d$ is the $d\times d$ identity matrix, and define $\bb{W}_i \circeq s \bb{U}_i - r \bb{V}_i$ so that the $j$-th component of $\bb{W}_i$ has a lattice distribution with span $H_{jj} = \text{gcd}(r,s)$.
    Note that $\bb{W}_i^{\star} \circeq \bb{H}^{-1} \bb{W}_i$ has span $1$ in all $d$ directions.
    The covariance matrix of $\bb{W}_i$ is given by $\bb{\Sigma}$ in \eqref{eq:lem:lemma.3.Leblanc.2006.tech.report.covariance.matrix}.
    We can write $S_{r,s,m}(\bb{x})$ in terms of the $\bb{W}_i^{\star}$'s as
    \begin{equation*}
        S_{r,s,m}(\bb{x}) = \PP\left(\sum_{i=1}^m s \bb{U}_i = \sum_{i=1}^m r \bb{V}_i\right) = \PP\left(\sum_{i=1}^m \bb{W}_i^{\star} = \bb{0}\right).
    \end{equation*}
    Therefore, using Theorem 3.1 of \cite{MR3475488} (a local central limit theorem for random vectors with lattice distributions), $\det(\bb{H}) = (\text{gcd}(r,s))^d$ and the fact that the covariance matrix of $\bb{W}_i^{\star}$ is equal to $\bb{H}^{-1} \bb{\Sigma} \bb{H}^{-1}$, we obtain the conclusion.
\end{proof}

The following proposition generalizes Lemma 4 in \cite{MR2662607} when $j = 0$.

\begin{proposition}\label{prop:lemma.4.Leblanc.2006.tech.report}
    Let $r,s,m \in \N$ and let $h : \mathcal{S}\to \R$ be any bounded measurable function. As $m\to\infty$,
    \begin{enumerate}[(a)]
        \item $m^{d/2} \int_{\mathcal{S}} S_{r,s,m}(\bb{x}) d\bb{x} = \frac{2^{-d} \sqrt{\pi}}{\Gamma(d/2 + 1/2)} + O(m^{-1}) = \int_{\mathcal{S}} \phi_{r,s}(\bb{x}) d\bb{x} + O(m^{-1})$,
        \item $\int_{\mathcal{S}} h(\bb{x}) (m^{d/2} S_{r,s,m}(\bb{x}) - \phi_{r,s}(\bb{x})) d\bb{x} = o(1)$.
    \end{enumerate}
\end{proposition}

\begin{proof}
    Assume for now that $r = s = 1$.
    We have
    \begin{align}\label{eq:prop:lemma.4.Leblanc.2006.tech.report.beginning}
        \int_{\mathcal{S}} S_{1,1,m}(\bb{x}) d\bb{x}
        &= \sum_{\|\bb{k}\| \leq m} \int_{\mathcal{S}} (P_{\bb{k},m}(\bb{x}))^2 d\bb{x}
        = \sum_{\|\bb{k}\| \leq m} \left(\frac{\Gamma(m + 1)}{\prod_{i=1}^{d+1} \Gamma(k_i + 1)}\right)^2 \int_{\mathcal{S}} \, \prod_{i=1}^{d+1} x_i^{2k_i} d\bb{x} \notag \\
        &= \sum_{\|\bb{k}\| \leq m} \left(\frac{\Gamma(m + 1)}{\prod_{i=1}^{d+1} \Gamma(k_i + 1)}\right)^2 \frac{\prod_{i=1}^{d+1} \Gamma(2k_i + 1)}{\Gamma(2m + d + 1)} \notag \\
        &= \frac{(\Gamma(m + 1))^2}{\Gamma(2m + d + 1)} \sum_{\|\bb{k}\| \leq m} \prod_{i=1}^{d+1} {2 k_i \choose k_i}.
    \end{align}
    To obtain the third equality, we used the normalization constant for the Dirichlet distribution. Note that
    \begin{align}\label{eq:prop:lemma.4.Leblanc.2006.tech.report.Graham.formula}
        \sum_{\|\bb{k}\| \leq m} \prod_{i=1}^{d+1} {2 k_i \choose k_i}
        &= (-4)^m \sum_{\|\bb{k}\| \leq m} \prod_{i=1}^{d+1} \frac{1}{(-4)^m} {2 k_i \choose k_i}
        = (-4)^m \sum_{\|\bb{k}\| \leq m} \prod_{i=1}^{d+1} {-1/2 \choose k_i} \notag \\
        &= (-4)^m {-(d+1)/2 \choose m} \notag \\
        &= {m + \frac{d-1}{2} \choose m} 4^m,
    \end{align}
    where the last three equalities follow, respectively, from (5.37), the Chu-Vandermonde convolution (p.\hspace{-1mm} 248), and (5.14) in \cite{MR1397498}.
    By applying \eqref{eq:prop:lemma.4.Leblanc.2006.tech.report.Graham.formula} and the duplication formula
    \begin{equation}\label{eq:prop:lemma.4.Leblanc.2006.tech.report.duplication.formula}
        4^y = \frac{2 \sqrt{\pi} \, \Gamma(2y)}{\Gamma(y) \Gamma(y + 1/2)}, \quad y\in (0,\infty),
    \end{equation}
    see \cite[p.256]{MR0167642}, in \eqref{eq:prop:lemma.4.Leblanc.2006.tech.report.beginning}, we get
    \begin{align*}
        \int_{\mathcal{S}} S_{1,1,m}(\bb{x}) d\bb{x}
        &= \frac{(\Gamma(m + 1))^2}{\Gamma(2m + d + 1)} \cdot \frac{\Gamma(m + d/2 + 1/2)}{\Gamma(m + 1)\Gamma(d/2 + 1/2)} \cdot 4^m \\
        &= \frac{2 \sqrt{\pi} \, \Gamma(m + 1)}{\Gamma(d/2 + 1/2) \Gamma(m + d/2 + 1)} \cdot \frac{\Gamma(m + d/2 + 1/2) \Gamma(m + d/2 + 1)}{2 \sqrt{\pi} \, \Gamma(2m + d + 1)} \cdot 4^m \\
        &= \frac{2 \sqrt{\pi} \, \Gamma(m + 1)}{\Gamma(d/2 + 1/2) \Gamma(m + d/2 + 1)} \cdot \frac{4^m}{4^{m + d/2 + 1/2}} \\
        &= \frac{2^{-d} \sqrt{\pi} \, \Gamma(m + 1)}{\Gamma(d/2 + 1/2) \Gamma(m + d/2 + 1)} \\
        &= \left\{\hspace{-1mm}
        \begin{array}{ll}
            \frac{2^{-d} \sqrt{\pi}}{\Gamma(d/2 + 1/2)} \prod_{i=1}^{d/2} (m + i)^{-1}, &\mbox{if } d ~\text{is even}, \\[2mm]
            \frac{2^{-d} \sqrt{\pi}}{\Gamma(d/2 + 1/2)} \prod_{i=1}^{d/2+1/2} (m + d/2 + 1 - i)^{-1} \cdot \frac{\Gamma(m + 1)}{\Gamma(m + 1/2)}, &\mbox{if } d ~\text{is odd}. \\
        \end{array}
        \right.
    \end{align*}
    Using the fact that
    \begin{equation}
        \frac{\Gamma(m + 1)}{m^{1/2} \Gamma(m + 1/2)} = 1 + \frac{1}{8m} + O(m^{-2}),
    \end{equation}
    see \cite[p.257]{MR0167642}, we obtain
    \begin{equation}\label{prop:lemma.4.Leblanc.2006.tech.report.S.m.asymptotic}
        m^{d/2} \int_{\mathcal{S}} S_{1,1,m}(\bb{x}) d\bb{x} = \frac{2^{-d} \sqrt{\pi}}{\Gamma(d/2 + 1/2)} + O(m^{-1}).
    \end{equation}
    In the case $r=s=1$, the expression for $\bb{\Sigma}$ in \eqref{eq:lem:lemma.3.Leblanc.2006.tech.report.covariance.matrix}
    is equal to $2 (\text{diag}(\bb{x}) - \bb{x}\bb{x}^T)$. Using the square-root-free symbolic Cholesky decomposition for covariance matrices of multinomial distributions (see Theorem 1 in \cite{MR1157720}), we deduce that $\det(\bb{\Sigma}) = 2^d \det(\text{diag}(\bb{x}) - \bb{x}\bb{x}^T) = 2^d \prod_{i=1}^{d+1} x_i$.
    Therefore,
    \begin{align}
        \int_{\mathcal{S}} \frac{1}{(2\pi)^{d/2} (\det(\bb{\Sigma}))^{1/2}} d\bb{x}
        &= \frac{1}{2^d \pi^{d/2}} \int_{\mathcal{S}} \prod_{i=1}^{d+1} x_i^{1/2 - 1} d\bb{x} \notag \\
        &= \frac{1}{2^d \pi^{d/2}} \cdot \frac{(\Gamma(1/2))^{d+1}}{\Gamma(d/2 + 1/2)} \notag  \\
        &= \frac{2^{-d} \sqrt{\pi}}{\Gamma(d/2 + 1/2)}.
    \end{align}
    Together with \eqref{prop:lemma.4.Leblanc.2006.tech.report.S.m.asymptotic} and \eqref{eq:lem:lemma.3.Leblanc.2006.tech.report.psi.function}, this proves $(a)$ for $r = s = 1$.

    Now, the almost-everywhere convergence from Lemma \ref{lem:lemma.3.Leblanc.2006.tech.report} and the mean convergence from $(a)$ imply that $\{S_{1,1,m}(\cdot)\}_{m\in \N}$ is uniformly integrable, see \cite[p.189]{MR1368405}.
    By Theorem \ref{thm:completely.monotonic}, $a\mapsto P_{a\bb{k},am}$ is decreasing on $(0,\infty)$, so
    \begin{equation}
        S_{r,s,m}(\bb{x}) \leq \sum_{\|\bb{k}\| \leq m} (P_{\bb{k},m}(\bb{x}))^2 = S_{1,1,m}(\bb{\bb{x}}),
    \end{equation}
    which implies that $\{S_{r,s,m}(\cdot)\}_{m\in \N}$ is also uniformly integrable.
    Hence, by Lemma \ref{lem:lemma.3.Leblanc.2006.tech.report}, we must have $(a)$ in the general case $r,s\in \N$.
    Finally, the almost-everywhere convergence and the uniform integrability imply the $L^1$ convergence, so $(b)$ follows immediately from Jensen's inequality and the fact that $h$ is bounded.
\end{proof}

\section*{Acknowledgements}

I would like to thank Alexandre Leblanc for promptly giving me access to \cite{Leblanc_2006_tech_report}.

%
%


\bibliographystyle{authordate1}
\bibliography{Ouimet_2018_multinomial_proba_bib}

\begin{thebibliography}{}

\bibitem[\protect\citename{Abramowitz \& Stegun, }1964]{MR0167642}
Abramowitz, M., \& Stegun, I.~A. 1964.
\newblock {\em Handbook of mathematical functions with formulas, graphs, and
  mathematical tables}.
\newblock National Bureau of Standards Applied Mathematics Series, vol. 55.
\newblock McGraw-Hill Book Company.
\newblock \href{http://www.ams.org/mathscinet-getitem?mr=MR0167642}{MR0167642}.

\bibitem[\protect\citename{Alzer, }2018]{MR3730425}
Alzer, H. 2018.
\newblock Complete monotonicity of a function related to the binomial
  probability.
\newblock {\em J. Math. Anal. Appl.}, {\bf 459}(1), 10--15.
\newblock \href{http://www.ams.org/mathscinet-getitem?mr=MR3730425}{MR3730425}.

\bibitem[\protect\citename{Athreya \& Janicki, }2016]{MR3475488}
Athreya, K.~B., \& Janicki, R. 2016.
\newblock Asymptotics of powers of binomial and multinomial probabilities.
\newblock {\em Statist. Probab. Lett.}, {\bf 112}, 58--62.
\newblock \href{http://www.ams.org/mathscinet-getitem?mr=MR3475488}{MR3475488}.

\bibitem[\protect\citename{Babu \& Chaubey, }2006]{MR2270097}
Babu, G.~J., \& Chaubey, Y.~P. 2006.
\newblock Smooth estimation of a distribution and density function on a
  hypercube using {B}ernstein polynomials for dependent random vectors.
\newblock {\em Statist. Probab. Lett.}, {\bf 76}(9), 959--969.
\newblock \href{http://www.ams.org/mathscinet-getitem?mr=MR2270097}{MR2270097}.

\bibitem[\protect\citename{Babu {\em et~al.}, }2002]{MR1910059}
Babu, G.~J., Canty, A.~J., \& Chaubey, Y.~P. 2002.
\newblock Application of {B}ernstein polynomials for smooth estimation of a
  distribution and density function.
\newblock {\em J. Statist. Plann. Inference}, {\bf 105}(2), 377--392.
\newblock \href{http://www.ams.org/mathscinet-getitem?mr=MR1910059}{MR1910059}.

\bibitem[\protect\citename{Belalia, }2016]{MR3474765}
Belalia, M. 2016.
\newblock On the asymptotic properties of the {B}ernstein estimator of the
  multivariate distribution function.
\newblock {\em Statist. Probab. Lett.}, {\bf 110}, 249--256.
\newblock \href{http://www.ams.org/mathscinet-getitem?mr=MR3474765}{MR3474765}.

\bibitem[\protect\citename{Belalia {\em et~al.}, }2017]{MR3630225}
Belalia, M., Bouezmarni, T., \& Leblanc, A. 2017.
\newblock Smooth conditional distribution estimators using {B}ernstein
  polynomials.
\newblock {\em Comput. Statist. Data Anal.}, {\bf 111}, 166--182.
\newblock \href{http://www.ams.org/mathscinet-getitem?mr=MR3630225}{MR3630225}.

\bibitem[\protect\citename{Dubourdieu, }1939]{MR0000436}
Dubourdieu, M.~J. 1939.
\newblock Sur un th\'eor\`eme de {M}. {S}. {B}ernstein relatif \`a la
  transformation de {L}aplace-{S}tieltjes.
\newblock {\em Compositio Math.}, {\bf 7}, 96--111.
\newblock \href{http://www.ams.org/mathscinet-getitem?mr=MR0000436}{MR0000436}.

\bibitem[\protect\citename{Ghosal, }2001]{MR1873330}
Ghosal, S. 2001.
\newblock Convergence rates for density estimation with {B}ernstein
  polynomials.
\newblock {\em Ann. Statist.}, {\bf 29}(5), 1264--1280.
\newblock \href{http://www.ams.org/mathscinet-getitem?mr=MR1873330}{MR1873330}.

\bibitem[\protect\citename{Graham {\em et~al.}, }1994]{MR1397498}
Graham, R.~L., Knuth, D.~E., \& Patashnik, O. 1994.
\newblock {\em Concrete mathematics}. Second edn.
\newblock Addison-Wesley Publishing Company, Reading, MA.
\newblock \href{http://www.ams.org/mathscinet-getitem?mr=MR1397498}{MR1397498}.

\bibitem[\protect\citename{Igarashi \& Kakizawa, }2014]{MR3174309}
Igarashi, G., \& Kakizawa, Y. 2014.
\newblock On improving convergence rate of {B}ernstein polynomial density
  estimator.
\newblock {\em J. Nonparametr. Stat.}, {\bf 26}(1), 61--84.
\newblock \href{http://www.ams.org/mathscinet-getitem?mr=MR3174309}{MR3174309}.

\bibitem[\protect\citename{Janssen {\em et~al.}, }2012]{MR2879763}
Janssen, P., Swanepoel, J., \& Veraverbeke, N. 2012.
\newblock Large sample behavior of the {B}ernstein copula estimator.
\newblock {\em J. Statist. Plann. Inference}, {\bf 142}(5), 1189--1197.
\newblock \href{http://www.ams.org/mathscinet-getitem?mr=MR2879763}{MR2879763}.

\bibitem[\protect\citename{Janssen {\em et~al.}, }2014]{MR3147339}
Janssen, P., Swanepoel, J., \& Veraverbeke, N. 2014.
\newblock A note on the asymptotic behavior of the {B}ernstein estimator of the
  copula density.
\newblock {\em J. Multivariate Anal.}, {\bf 124}, 480--487.
\newblock \href{http://www.ams.org/mathscinet-getitem?mr=MR3147339}{MR3147339}.

\bibitem[\protect\citename{Janssen {\em et~al.}, }2017]{MR3668546}
Janssen, P., Swanepoel, J., \& Veraverbeke, N. 2017.
\newblock Smooth copula-based estimation of the conditional density function
  with a single covariate.
\newblock {\em J. Multivariate Anal.}, {\bf 159}, 39--48.
\newblock \href{http://www.ams.org/mathscinet-getitem?mr=MR3668546}{MR3668546}.

\bibitem[\protect\citename{Kakizawa, }2011]{MR2769276}
Kakizawa, Y. 2011.
\newblock A note on generalized {B}ernstein polynomial density estimators.
\newblock {\em Stat. Methodol.}, {\bf 8}(2), 136--153.
\newblock \href{http://www.ams.org/mathscinet-getitem?mr=MR2769276}{MR2769276}.

\bibitem[\protect\citename{Lange, }1995]{Lange_1995}
Lange, K. 1995.
\newblock Applications of the Dirichlet distribution to forensic match
  probabilities.
\newblock {\em Genetica}, {\bf 96}(1-2), 107--117.
\newblock \href{https://doi.org/10.1007/BF01441156}{doi:10.1007/BF01441156}.

\bibitem[\protect\citename{Leblanc, }2006]{Leblanc_2006_tech_report}
Leblanc, A. 2006.
\newblock A bias-corrected approach to density estimation using {B}ernstein
  polynomials.
\newblock {\em Technical Report},  1--24.
\newblock University of Manitoba, Dept. of Statistics.

\bibitem[\protect\citename{Leblanc, }2009]{MR2488150}
Leblanc, A. 2009.
\newblock Chung-{S}mirnov property for {B}ernstein estimators of distribution
  functions.
\newblock {\em J. Nonparametr. Stat.}, {\bf 21}(2), 133--142.
\newblock \href{http://www.ams.org/mathscinet-getitem?mr=MR2488150}{MR2488150}.

\bibitem[\protect\citename{Leblanc, }2010]{MR2662607}
Leblanc, A. 2010.
\newblock A bias-reduced approach to density estimation using {B}ernstein
  polynomials.
\newblock {\em J. Nonparametr. Stat.}, {\bf 22}(3-4), 459--475.
\newblock \href{http://www.ams.org/mathscinet-getitem?mr=MR2662607}{MR2662607}.

\bibitem[\protect\citename{Leblanc, }2012a]{MR2960952}
Leblanc, A. 2012a.
\newblock On estimating distribution functions using {B}ernstein polynomials.
\newblock {\em Ann. Inst. Statist. Math.}, {\bf 64}(5), 919--943.
\newblock \href{http://www.ams.org/mathscinet-getitem?mr=MR2960952}{MR2960952}.

\bibitem[\protect\citename{Leblanc, }2012b]{MR2925964}
Leblanc, A. 2012b.
\newblock On the boundary properties of {B}ernstein polynomial estimators of
  density and distribution functions.
\newblock {\em J. Statist. Plann. Inference}, {\bf 142}(10), 2762--2778.
\newblock \href{http://www.ams.org/mathscinet-getitem?mr=MR2925964}{MR2925964}.

\bibitem[\protect\citename{Leblanc \& Johnson, }2007]{MR2345922}
Leblanc, A., \& Johnson, B.~C. 2007.
\newblock On a uniformly integrable family of polynomials defined on the unit
  interval.
\newblock {\em JIPAM. J. Inequal. Pure Appl. Math.}, {\bf 8}(3), Article 67, 5
  pp.
\newblock \href{http://www.ams.org/mathscinet-getitem?mr=MR2345922}{MR2345922}.

\bibitem[\protect\citename{Lu, }2015]{MR3412755}
Lu, L. 2015.
\newblock On the uniform consistency of the {B}ernstein density estimator.
\newblock {\em Statist. Probab. Lett.}, {\bf 107}, 52--61.
\newblock \href{http://www.ams.org/mathscinet-getitem?mr=MR3412755}{MR3412755}.

\bibitem[\protect\citename{Miller \& Samko, }2001]{MR1872377}
Miller, K.~S., \& Samko, S.~G. 2001.
\newblock Completely monotonic functions.
\newblock {\em Integral Transform. Spec. Funct.}, {\bf 12}(4), 389--402.
\newblock \href{http://www.ams.org/mathscinet-getitem?mr=MR1872377}{MR1872377}.

\bibitem[\protect\citename{Petrone, }1999]{MR1703623}
Petrone, S. 1999.
\newblock Bayesian density estimation using {B}ernstein polynomials.
\newblock {\em Canad. J. Statist.}, {\bf 27}(1), 105--126.
\newblock \href{http://www.ams.org/mathscinet-getitem?mr=MR1703623}{MR1703623}.

\bibitem[\protect\citename{Prakasa~Rao, }2005]{MR2153833}
Prakasa~Rao, B. L.~S. 2005.
\newblock Estimation of distribution and density functions by generalized
  {B}ernstein polynomials.
\newblock {\em Indian J. Pure Appl. Math.}, {\bf 36}(2), 63--88.
\newblock \href{http://www.ams.org/mathscinet-getitem?mr=MR2153833}{MR2153833}.

\bibitem[\protect\citename{Qi {\em et~al.}, }2018]{hal-01769288}
Qi, F., Niu, D.-W., Lim, D., \& Guo, B.-N. 2018.
\newblock Some logarithmically completely monotonic functions and inequalities
  for multinomial coefficients and multivariate beta functions.
\newblock {\em Preprint. HAL archives-ouvertes.},  1--13.
\newblock \href{https://hal.archives-ouvertes.fr/hal-01769288}{hal-01769288}.

\bibitem[\protect\citename{Shiryaev, }1996]{MR1368405}
Shiryaev, A.~N. 1996.
\newblock {\em Probability}. Second edn.
\newblock Graduate Texts in Mathematics, vol. 95.
\newblock Springer-Verlag, New York.
\newblock \href{http://www.ams.org/mathscinet-getitem?mr=MR1368405}{MR1368405}.

\bibitem[\protect\citename{Tanabe \& Sagae, }1992]{MR1157720}
Tanabe, K., \& Sagae, M. 1992.
\newblock An exact {C}holesky decomposition and the generalized inverse of the
  variance-covariance matrix of the multinomial distribution, with
  applications.
\newblock {\em J. Roy. Statist. Soc. Ser. B}, {\bf 54}(1), 211--219.
\newblock \href{http://www.ams.org/mathscinet-getitem?mr=MR1157720}{MR1157720}.

\bibitem[\protect\citename{Tenbusch, }1994]{MR1293514}
Tenbusch, A. 1994.
\newblock Two-dimensional {B}ernstein polynomial density estimators.
\newblock {\em Metrika}, {\bf 41}(3-4), 233--253.
\newblock \href{http://www.ams.org/mathscinet-getitem?mr=MR1293514}{MR1293514}.

\bibitem[\protect\citename{van Haeringen, }1996]{MR1421454}
van Haeringen, H. 1996.
\newblock Completely monotonic and related functions.
\newblock {\em J. Math. Anal. Appl.}, {\bf 204}(2), 389--408.
\newblock \href{http://www.ams.org/mathscinet-getitem?mr=MR1421454}{MR1421454}.

\bibitem[\protect\citename{Vitale, }1975]{MR0397977}
Vitale, R.~A. 1975.
\newblock Bernstein polynomial approach to density function estimation.
\newblock {\em Pages  87--99 of:} {\em Statistical inference and related topics
  ({P}roc. {S}ummer {R}es. {I}nst. {S}tatist. {I}nference for {S}tochastic
  {P}rocesses, {I}ndiana {U}niv., {B}loomington, {I}nd., 1974, {V}ol. 2;
  dedicated to {Z}. {W}. {B}irnbaum)}.
\newblock Academic Press, New York.
\newblock \href{http://www.ams.org/mathscinet-getitem?mr=MR0397977}{MR0397977}.

\end{thebibliography}

\end{document}